\documentclass[a4paper,12pt,reqno,twoside]{amsart}
\pdfoutput=1
\usepackage{amsmath}
\usepackage{amsfonts}
\usepackage{amssymb}
\usepackage{amsthm}
\usepackage{color}
\usepackage{ifpdf}
\usepackage{array}

\usepackage{url}
\usepackage{graphicx}
\usepackage{float}
\usepackage{multirow,bigdelim}

\numberwithin{equation}{section}
\newtheorem{definition}{Definition}
\newtheorem{theorem}{Theorem}
\newtheorem{lemma}{Lemma}

\usepackage{hyperref}
\usepackage{cleveref}
\usepackage{lipsum}
\crefname{lemma}{Lemma}{Lemmas}

\newcommand*{\C}{\mathbb{C}}
\newcommand*{\R}{\mathbb{R}}
\newcommand*{\N}{\mathbb{N}}

\newcommand{\comment}[1]{}
\title[P\'{o}lya-like approximation]%
      {Real-rooted P\'{o}lya-like approximations to the Riemann Xi-function} 
\author[Y. SHI]{Yaoming SHI}
\date{Version of \today}
\AtBeginDocument{%
\begin{abstract}
The Riemann $\Xi(z)$ function admits a Fourier transform of a even kernel $\Phi(t)$. The latter is related to the derivatives of Jacobi theta function $\theta(z)$, a modular form of weight $1/2$. P\'{o}lya noticed that when $t$ goes to infinity, $e^t$ goes to $e^t+ e^{-t}=2\cosh t$. He then approximated the kernel $\Phi(t)$ by $\Phi_{P}(t)$ that contained only the leading term and with $\exp t,\exp(9t/4)$ replaced by $2\cosh t,2\cos(9t/4)$. This procedure captured almost all of the contribution from the tail part (i.e., $t\to\infty$) of the kernel $\Phi(t)$. 

\indent 
We realize that when $t$ goes to infinity and $0\leqslant b<1,c\in\R$, $\cosh t+c \cosh(bt)$ goes to $\cosh t$. Thus we improve P\'{o}lya's approximation by replacing $\cosh(9t/4)$ with $\cosh(9t/4)+b\sum_{k=0}^{m-1}b_k \cosh(9kt/(4m))$ and adjusting the parameters $b,b_k,m$ such that (A) the approximated kernel $\Phi_{S}(b,b_k,m;t)$ goes to $\Phi(t)$when $t$ goes to infinity;(B) $\Phi_{S}(b,b_k,m;t)$ is identical to $\Phi(t)$ at $t=0$; (C) the Fourier transform of $\Phi_{S}(b,b_k,m;t)$,like in P\'{o}lya's case, has only real zeros.  Since this procedure also captures almost all of the contribution from the head part (i.e., near $t=0$) of the kernel $\Phi(t)$, we are able to anchor both ends of the kernel $\Phi(t)$.
\end{abstract}
\maketitle
}
\begin{document}

%

\footnotetext{yaoming\_shi@yahoo.com}
\section{Introduction}

Let $s,z$ be two complex variables, 
$\zeta(s)$ be the Riemann $\zeta$-function,  
\begin{equation}
\xi(s)=\frac{1}{2}s(s-1)\pi^{-s/2}\Gamma\left(\frac{s}{2}\right)\zeta(s)
\end{equation}
be the Riemann (lower case) $\xi$-function, and 
\begin{equation}
\Xi(z)=\xi(iz+1/2)
\end{equation}
be the Riemann (upper case) $\Xi$-function, which is an entire function~\cite{17,6} 
satisfying functional equation $\Xi(z)=\Xi(-z)$ and $\Xi(\bar{z})=\overline{\Xi(z)}$. Riemann hypothesis~\cite{1,2} is then equivalent to the statement that all the zeros of $\Xi(z)$ are real.

Riemann $\Xi(2z)$ function can be expressed as a Fourier transformation ~\cite{16,17,6}:
\begin{equation}
\Xi(2z)=\int_{-\infty}^{\infty}\Phi(t)\exp(izt){\rm d}u=2\int_0^{\infty}\Phi(t)\cos(zt){\rm d}t,
\end{equation}
where
\begin{equation}\label{Phi}
\aligned
\Phi(t)&=2e^{9t/4}\theta''(e^{t})+3e^{5t/4}\theta'(e^{t})\\
&=\sum_{n=1}^{\infty}\left(2\pi^2n^4e^{9t/4}-3\pi n^2e^{5t/4}\right)\exp\left(-\pi n^2 e^{t}\right)\\
&=\Phi(-t).
\endaligned
\end{equation}
\noindent And where $\theta(x)$ is the Jacobi theta function defined below in \eqref{thetay}.

As summarized by Dimitrov ~\cite{3} and with Rusev ~\cite{5},then a natural approach to resolving the Riemann hypothesis is to establish criteria for an entire function, or more specifically, a Fourier transform of a kernel, to possess only real zeros and to apply them to the Riemann $\Xi(z)$ function. There is no doubt that this was the main reason why so many celebrated mathematicians have been interested in the zero distribution of entire functions and, in particular, of Fourier transforms. Among them are
such distinguished masters of the Classical Analysis as A. Hurwitz, J.L.W.
V. Jensen, G. P´olya, H.G. Hardy, E. Tichmarsh, W. de Bruin, Newman, N. Obrechkoff, L. Tchakaloff etc.
 
For complete review we refer the readers to the excellent and 108 page review paper by Dimitrov and Rusev ~\cite{5}. See also the review paper by Ki ~\cite{9} and Hallum's 2014 Master Thesis ~\cite{7}(an easy-to-read reference).

\indent 
P\'{o}lya noticed that when $t$ goes to infinity, $\exp (at)\to \exp (at)+ \exp(-at)=2\cosh (at)$. He then approximated the kernel $\Phi(t)$ by $\Phi_{P}(t)$ that contained only the leading term and with $\exp t, \exp(9t/4)$ replaced by $2\cosh t,2\cosh (9t/4)$. We realize that when $t$ goes to infinity and $0<b<1,m\in\N$, $h(t)=\cosh(t/4)(4\cosh^2(t/m)-4b^2)^{m}$ goes to $\cosh (9t/4)$. Thus we improve P\'{o}lya's approximation by replacing $\cosh(9t/4)$ with $h(t)$ and adjust the parameters $b,m$ such that (A) the approximated kernel $\Phi_{S3}(b,m;t)$ goes to $\Phi(t)$when $t$ goes to infinity;(B) $\Phi_{S3}(b,m;t)$ is identical to $\Phi(t)$ at $t=0$; (C) the Fourier transform of $\Phi_{S3}(b,m;t)$, has only real zeros. Since this procedure also captures almost all of the contribution from the head part (i.e., near $t=0$) of the kernel $\Phi(t)$, we are able to anchor both ends of the kernel $\Phi(t)$. It remains to see if one can better approximate the body of $\Phi(t)$. 

Thus our criteria for picking kernel $K(t)$ to approximate the kernel $\Phi(t)$ of \eqref{Phi} are

(i) $K(t)\to \Phi(t),\quad t\to \infty$,

(ii) $K(0)=\Phi(0)$, 

(iii) $\int_0^{\infty}K(t)\cos(z t)\mathrm{d}t$ has only real zeros.

Using these criteria, we obtain several new and improved approximations to kernel $\Phi(t)$ and find out that their Fourier transforms have only real zeros. 

Here is the outline of the paper. We introduce notations and necessary lemmas in section 2. The approximations to $\Phi(t)$ by P\'{o}lya, de Bruijin, and Hejhal related to this paper are introduced in section 3. We also plot these approximated Phi functions and their corresponding Fourier transforms.  We present our new approximations to $\Phi(t)$ in section 4. Figures of these new approximated Phi functions and their Fourier transforms are readily compared to those in section 3.  In section 5, we provide conclusion.

\section{Notations and Definitions}
Almost all of the material in this section can be found in \cite{7,5} or references therein.
\noindent The complex function
\begin{equation}
\theta_3(z, \tau) :=\sum_{n=-\infty}^{\infty}\exp(i\pi n^2\tau + 2i\pi nz), z \in \C, \text{Im}\tau > 0,
\end{equation}
is one of the Jacobi theta-functions ~\cite{10,18}. The function $\theta_3(0, \tau)$ is holomorphic in the upper half-plane ($\text{Im}\tau > 0$) and satisfies the relations 
\begin{equation}
\theta_3(0, \tau+1) = \theta_3(0, \tau-1),
\end{equation}
\begin{equation}
\theta_3(0,-1/\tau) =(-i\tau)^{1/2}\theta_3(0, \tau), 
\end{equation}
where $(-i\tau)^{1/2} := \exp((1/2) \log(-i\tau))$. Thus $\theta_3(0, \tau)$ is a modular form of weight $1/2$.

\noindent For simplification, a Jacobi $\theta(x)$ function is often defined by setting $\tau=i x$ in $\theta_3(0, \tau)$ as:
\begin{equation}\label{thetay}
\theta(x):=\theta_3(0, i x) =\sum_{n=-\infty}^{\infty}\exp(-\pi n^2 x ), \qquad x > 0.
\end{equation}

\noindent It satisfies relations
\begin{equation}\label{thetaS1}
\theta(x) = \theta(x+2i),
\end{equation}
\begin{equation}\label{thetaS2}
\theta(1/x) =x^{1/2}\theta(x). 
\end{equation}

\begin{definition}
A real entire function $f(z)=\sum_{k=0}^{\infty}\frac{\gamma_k}{k!}z^k$ is in the Laguerre-P\'{o}lya class, written $f(z)\in \mathcal{LP}$, if
\begin{equation}
f(z)=cz^m\exp(-az^2+bz)\prod\limits_{k=1}^{\omega}\left(1+\frac{z}{z_k}\right)\exp(-z/z_k)
\end{equation}
where $b,c,z_k\in\R$(i.e.,all the zeros are real),$m\in\N_0,a\ge0,0\leqslant\omega\leqslant\infty$ and $\sum_{k=1}^{\omega}\frac{1}{z_k^2}<\infty$.
\end{definition}

G. P\'{o}lya ~\cite{14} introduced a class of functions he termed universal factors. Let $K(t)$ be an even and real-valued function that is absolutely integrable over $R$.
Also, suppose, for $b > 2$,$K(t) = O(\exp(-|t|^b)),t\to \pm \infty$.
\begin{definition}
Universal factors are the collection of functions, $\{\phi(t)\}$, such that if the integral
\begin{equation}
\int_{-\infty}^{\infty}K(t)\exp(izt) \mathrm{d}t \in \mathcal{LP},
\end{equation}
then the integral
\begin{equation}
\int_{-\infty}^{\infty}\phi(t)K(t)\exp(izt) \mathrm{d}t \in \mathcal{LP}.
\end{equation}
\end{definition}
G. P\'{o}lya was able to completely characterize the functions, $\phi(t)$, that comprise this class.
\begin{lemma}[P\'{o}lya's Universal Factor Theorem]\label{ufactor}
If $\phi(iz) \in \mathcal{LP}$, then $\phi(t)$ is a universal factor.
If the real analytic function $\phi(t)$ is a universal factor, then $\phi(it) \in \mathcal{LP}$.
\end{lemma}
\begin{lemma}[Enestr\"{O}m-Kakeya Theorem]\label{Enestrom_Kakeya}
If $0 < a_0 < a_1 < a_2 <\cdots < a_n$, then the polynomial $p_n(z) =\sum_{k=0}^{n} a_k z_k$ has all of its zeros in the closed unit disk $D = {z : |z| < 1}$.
\end{lemma}
\begin{lemma}[Hermite-Biehler theorem]\label{Hermite_Biehler}
If the zeros of the algebraic polynomial with complex
coefficients $p_n(z) = \sum_{k=0}^n c_kz_k$ belong to unit disk $D$ and if we set $z = \cos \alpha + i \sin \alpha$ and separate the real and the imaginary parts,$p_n(z) = A(\alpha) + iB(\alpha)$, then the trigonometric polynomials $A(\alpha)\in \mathcal{LP}$ and $B(\alpha)\in \mathcal{LP}$ and their zeros interlace.
\end{lemma}
The ~\cref{Enestrom_Kakeya} and ~\cref{Hermite_Biehler} already imply 
\begin{lemma}
If $0 < a_0 < a_1 < a_2 <\cdots < a_n$, then
\begin{equation}\label{eqpcos}
A(\alpha) = \sum_{k=0}^{n} a_k \cos (k\alpha) \in \mathcal{LP}, 
\end{equation}
\begin{equation}\label{eqpsin}
B(\alpha) = \sum_{k=1}^{n} a_k \sin (k\alpha) \in \mathcal{LP},
\end{equation}
and their zeros are interlace.
Thus $B(it),t\in \R$ is a universal factor.
\end{lemma}
Let $K_{z}(a)$ be the Modified Bessel function of the second:
\begin{equation}
K_{z}(a)=\int_{0}^{\infty}\exp(-a\cosh t)\cosh(tz)\mathrm{d}t=K_{-z}(a).
\end{equation}

\begin{lemma}\label{Kiz}~\cite{13,14}
\begin{equation}
K_{iz}(2\pi)=K_{-iz}(2\pi) \in \mathcal{LP}.
\end{equation}
\end{lemma}

\begin{lemma}\label{Kizplusc}~\cite{13,14}
Let $A,c>0$, then
\begin{equation}
K_{iz+c}(A)+ K_{iz-c}(A)\in \mathcal{LP}.
\end{equation}
\end{lemma}

\section{Approximations to $\Phi(t)$ by P\'{o}lya, de Bruijin, and Hejhal}
\indent P\'{o}lya ~\cite{13,14} approximated $\Phi(t)$ with $\Phi_{P}(t)$ and $\Phi_{P2}(t)$ by keeping only the leading ($n=1$) term in \eqref{Phi} and replaced $e^{at}$ with $(e^{at}+e^{-at})=2\cosh(at)$:
\begin{equation}\label{PhiP}
\Phi_{P}(t)=4\pi^2\cosh(9t/4)\exp\left(-2\pi \cosh t\right),
\end{equation}
\begin{equation}\label{PhiP2}
\Phi_{P2}(t)=\left(4\pi^2\cosh(9t/4)-6\pi\cosh(5t/4)\right)\exp\left(-2\pi \cosh t\right).
\end{equation}
Thus when $t\to\infty$, $\Phi(t)\to\Phi_{P}(t),\Phi_{P2}(t)\to\Phi_{P}(t)$.
The Fourier transforms of $\Phi_{P}(t)$ and $\Phi_{P2}(t)$, are given by:
\begin{equation}
\Xi_{P}(2z)=4\pi^2\left(K_{iz+9/4}(2\pi)+K_{iz-9/4}(2\pi)\right),
\end{equation}
\begin{equation}
\aligned
\Xi_{P2}(2z)&=4\pi^2\left(K_{iz+9/4}(2\pi)+K_{iz-9/4}(2\pi)\right)\\
&-6\pi\left(K_{iz+5/4}(2\pi)+K_{iz-5/4}(2\pi)\right).
\endaligned
\end{equation}
\noindent P\'{o}lya proved that $\Xi_{P}(2z)$ and $\Xi_{P2}(2z)$ have only real zeros.

\noindent de Bruijn ~\cite{4} approximated $\Phi(t)$ with $\Phi_{dB}(t)$:
\begin{equation}\label{PhidB}
\aligned
\Phi_{dB}(t)&=\exp\left(-2\pi \cosh t\right)\\
&\times\left(4\pi^2\cosh(t/4)+(4\pi^3-6\pi)\cosh(5t/4)+4\pi^2\cosh(9t/4)\right)
\endaligned
\end{equation}
The Fourier transform of $\Phi_{dB}(t)$ is given by:
\begin{equation}
\aligned
\Xi_{dB}(2z)&=4\pi^2\left(K_{iz+9/4}(2\pi)
+K_{iz-9/4}(2\pi)\right)\\
&+(4\pi^3-6\pi)\left(K_{iz+5/4}(2\pi)
+K_{iz-5/4}(2\pi)\right)\\
&+4\pi^2\left(K_{iz+1/4}(2\pi)+K_{iz-1/4}(2\pi)\right),
\endaligned
\end{equation}
de Bruijn proved that the function $\Xi_{dB}(2z)$ has only real zeros.

\noindent Hejhal ~\cite{8} approximated $\Phi(t)$ with $\Phi_{H,m}(t)$:
\begin{equation}\label{PhiH}
\Phi_{H,m}(t)=\sum_{n=1}^{m}\left(4\pi^2n^4\cosh(9t/4)-6\pi n^2\cosh(5t/4)\right)\exp\left(-2\pi n^2 \cosh t\right)
\end{equation}
The resulting $\Xi_{H,m}(z)$ is given by:
\begin{equation}
\aligned
\Xi_{H,m}(2z)=\sum_{n=1}^{m}4n^4\pi^2\left(K_{iz+9/4}(2\pi n^2)+K_{iz-9/4}(2\pi n^2)\right)\\
-\sum_{n=1}^{m}6n^2\pi\left(K_{iz+5/4}(2\pi n^2)+K_{iz-5/4}(2\pi n^2)\right).
\endaligned
\end{equation}
Clearly when $m\to\infty$, $\Phi_{H,m}(t) \not\to \Phi(t)$, and $\Xi_{H,m}(2z) \not\to \Xi(2z)$. Thus the study of this general approximation is often considered not to be directly related to a possible proof of the Riemann hypothesis. Nevertheless Hejhal proved that almost all the zeros of the function $\Xi_{H,m}(2z)$ are real. 

We notice that there is one thing in common in P\'{o}lya's approximation $\Phi_{P}(t), \Phi_{P2}(t)$ of \eqref{PhiP} and \eqref{PhiP2}, de Bruijin's approximation $\Phi_{dB}(t)$ of \eqref{PhidB}, and Hejhal's approximation $\Phi_{H,m}(t)$ of \eqref{PhiH}, that they all captured the contribution of the tail part (at $t\to\infty$) of $\Phi(t)$ in the Fourier transformation. But none of them converges to $\Phi(t)$ near $t=0$.  This aspect is clearly shown in Figure 1. Thus they can hardly capture the contribution of the head part (at $t=0$) of $\Phi(t)$ in the Fourier transformation. 
%
\begin{figure}[H]
\centering
\includegraphics[scale=0.8,keepaspectratio]{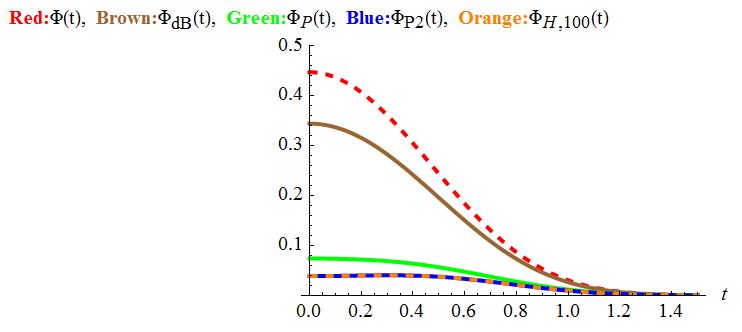}
\caption{Plots of various Phi functions vs. $t$. This includes $\Phi(t)$(Red), de bruijin's $\Phi_{D}(t)$(Brown), P\'{o}lya's $\Phi_{P}(t)$(Green) and $\Phi_{P2}(t)$(Blue), and Hejhal's $\Phi_{H,100}(t)$(Black). There is no visible difference between $\Phi_{H,10}(t)$ and $\Phi_{H,100}(t)$. }\label{figure1}
\end{figure}

A natural question then arises: Is it possible to find approximations to $\Phi(t)$ such that they converge to $\Phi(t)$ at $t\to\infty$ and $t=0$, and the corresponding Fourier transforms have only real zeros ?
We will give positive answers to this question in the next subsection.

In Figures 2 below we compare $\Xi_{P}(z)$(Blue) against $\Xi(z)$(Red). It showed that there existed 29 zeros for both $\Xi_{P}(z)$ and $\Xi(z)$. 
%
\begin{figure}[H]
\centering
\includegraphics[scale=0.8,keepaspectratio]{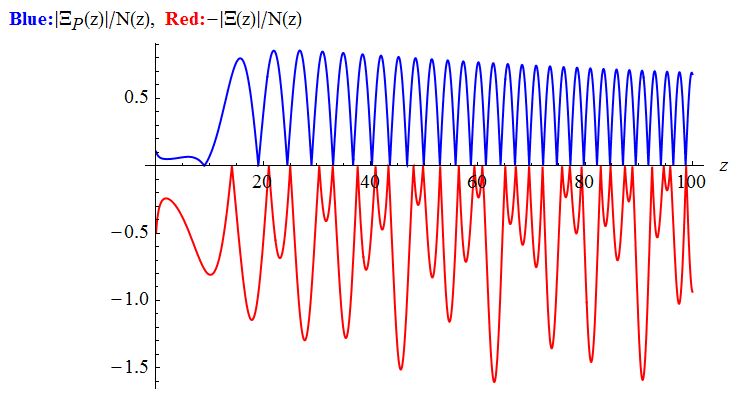}
\caption{Plots of various Xi functions vs. $z$. This includes: $|\Xi_{P}(z)|/N(z)$(Blue); $-|\Xi(z)|/N(z)$(Red). Here and after $N(z)=\exp(-\pi z/4)(z+1)^2$ is a scale normalization function.}
\end{figure}
In Figures 3 below we compare $\Xi_{P2}(z)$(Purple) against $\Xi(z)$(Red). It showed that there existed 29 zeros for both $\Xi_{P2}(z)$ and $\Xi(z)$. 
%
\begin{figure}[H]
\centering
\includegraphics[scale=0.8,keepaspectratio]{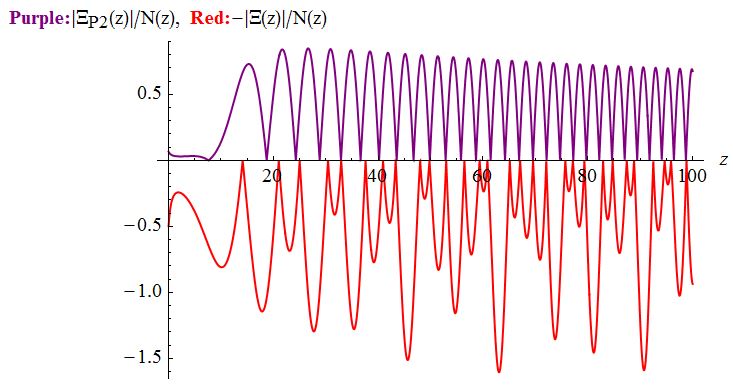}
\caption{Plots of various Xi functions vs. $z$. This includes: $|\Xi_{P2}(z)|/N(z)$(Purple); $-|\Xi(z)|/N(z)$(Red).}
\end{figure}
In Figures 4 below we compare $\Xi_{dB}(z)$(Brown) against $\Xi(z)$(Red). It showed that there existed 29 zeros for both $\Xi_{dB}(z)$ and $\Xi(z)$. 
%
\begin{figure}[H]
\centering
\includegraphics[scale=0.8,keepaspectratio]{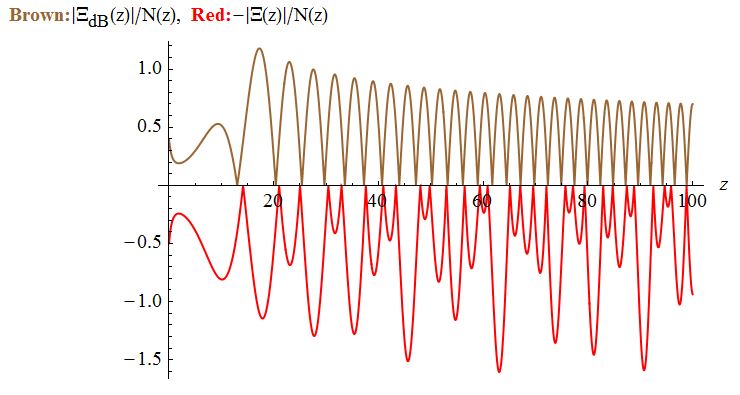}
\caption{Plots of various Xi functions vs. $z$. This includes: $|\Xi_{dB}(z)|/N(z)$(Brown); $-|\Xi(z)|/N(z)$(Red).}
\end{figure}
In Figures 5 below we compare $\Xi_{H,1}(z)$(Green) and $\Xi_{H,4}(z)$(Blue) against $\Xi(z)$(Red). It showed that there existed 29 zeros for $\Xi_{H,1}(z)$, $\Xi_{H,4}(z)$, and $\Xi(z)$. 
%
\begin{figure}[H]
\centering
\includegraphics[scale=0.8,keepaspectratio]{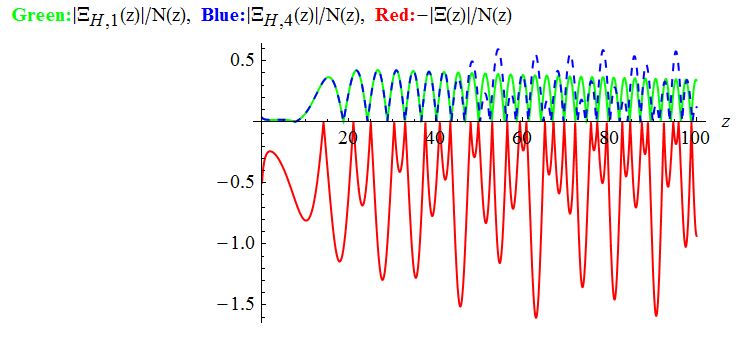}
\caption{Plots of various Xi functions vs. $z$. This includes: $|\Xi_{H,1}(z)|/N(z)$(Green),$|\Xi_{H,4}(z)|/N(z)$(Blue);, $-|\Xi(z)|/|N(z)|$(Red).}
\end{figure}

\section{Our Apprixmiations to the kernel $\Phi(t)$}

%
%
Let $\lambda>0,\overline{\alpha}_{-n}=\alpha_n$ and 
\begin{equation}
f(t)=\exp(-\lambda \cosh t)\sum_{n=-N}^{N}\alpha_n \exp(nt),
\end{equation}

\begin{equation}\label{Psiz}
\Psi(z)=\int_{-\infty}^{\infty}f(t)\exp(izt)\mathrm{d}t.
\end{equation}

de Bruijn proved that $\Psi(z)$ of \eqref{Psiz} has $N$ pair of non-real zeros at most ~\cite{4}(Theorem 21). de Bruijn commented that function $\Psi(z)$ may be of some interest since the Riemann Xi-function can be approximated by functions of this type.

We would like to point out that with $a_{-n}=a_{n}\in \R,\lambda=2\pi$, when $t\to\infty$, because $N$ is a positive integer, $\cosh(Nt)\not =\cosh(9t/4)$, so
\begin{equation}
f(t)\to 2a_N\cosh(Nt)\exp\left(-2\pi \cosh t\right)\not =4\pi^2\cosh(9t/4)\exp\left(-2\pi \cosh t\right)=\Phi_{P}(t)
\end{equation}
Thus $f(t)$ does not have the proper behavior near $t\to \infty$. But we can remedy this problem.
Let $\lambda=2\pi,\alpha_{-n}=\alpha_n\in\R,0\leqslant\beta_{-n}=\beta_n<1$ and 
\begin{equation}\label{FdB}
K(t)=\Phi_P(t)+\exp(-2\pi \cosh t)\sum_{n=-N}^{N}\alpha_n \exp(9\beta_n t/4).
\end{equation}
So when $t\to\infty$,$K(t)\to \Phi_{P}(t)$.
Thus criteria (i)$K(t)\to \Phi(t),t\to \infty$, mentioned in the introduction, is satisfied. 
The actual values of parameters $\alpha_n$ and $\beta_n$ are then used to satisfy the other two criteria; namely (ii) $K(0)=\Phi(0)$, (iii) $\int_0^{\infty}K(t)\cos(z t)\mathrm{d}t$ has only real zeros.

In theory one can also use the following $K(t)$ to approximate $\Phi(t)$.
\begin{equation}\label{FdBint}
\aligned
K(t)=\Phi_P(t)&+\exp(-2\pi \cosh t)\sum_{n=-N}^{N}\alpha_n \exp(9\beta_n t/4)\\
&+\exp(-2\pi \cosh t)\int_{-A}^{A}\gamma(\mu) \exp(9\delta(\mu) t/4)\mathrm{d}t.
\endaligned
\end{equation}
Where $A>0,\mu\in\R,\gamma(-\mu)=\gamma(\mu)\in\R,\delta(\mu)=\delta(\mu)\in\R$.

In all of our approximations below, we will use the $K(t)$ of the type \eqref{FdB} to approximate $\Phi(t)$.

Let
\begin{equation}
\Phi(0)=\theta^{''}(1)+(3/2)\theta^{'}(1)\approx 0.446696
\end{equation}
where $\theta(x)$ is defined in \eqref{thetay}.

We first approximate $\Phi(t)$ with $\Phi_{S}(t)$.
\begin{theorem}
Let 
\begin{equation}\label{PhiSm}
\Phi_{S}(m;t)=4\pi^2 f_m(t)\exp\left(-2\pi \cosh t\right),
\end{equation}
\begin{equation}
f_m(t)=\cosh\left(9t/4\right)+b \sum_{k=0}^{m-1}b_k \cosh\left(9kt/(4m)\right),
\end{equation}
where $0<b_k=\frac{k+1}{m+1}<1, k=0,1,\dots m-1$.
\noindent The Fourier transform of $\Phi_{S}(m;t)$ is:
\begin{equation}\label{XiSm}
\aligned
\Xi_{S}(m;2z)=4\pi^2\left(G_{2\pi}(iz+9/4)+G_{2\pi}(iz-9/4)\right)\\
+4\pi^2 b \sum_{k=0}^{m-1}b_k(G_{2\pi}(iz+9k/(4m))+G_{2\pi}(iz-9k/(4m))),
\endaligned
\end{equation}
where $G_{a}(z):=K_{z}(a)$.

If $m\ge 11$ and $b$ is determined by 
\begin{equation}\label{bm}
b m =2\beta,\qquad \beta:=(4\pi^2)^{-1}(e^{2\pi}\Phi(0)-1)\approx 5.059069
\end{equation}
\noindent then 

(A)$\Phi_{S}(m;t)\to \Phi(t),\text{when}\quad t\to\infty$;

(B)$\Phi_{S}(m;0)= \Phi(0)$; 

(C)the entire function $\Xi_{S}(m;2z)$ has only real zeros.
\end{theorem}
\begin{proof}
Since $0<\frac{k}{m}<1$, $f_m(t)\to \cos(9t/4)$ and $\Phi_{S}(m;t)\to \Phi(t)$ when $t\to\infty$. Thus we proved (A). Setting $\Phi_{S}(m;0)=\Phi(0)$ leads to \eqref{bm}, thus we proved (B).
 
Becuase of \cref{Kiz},it suffice to prove that $f_m(t)$ is a universal factor, or $f_m(it)\in \mathcal{LP}$. Defining $x=\frac{9t}{4m}$, then we obtain:
\begin{equation}
f_m(it)=\tilde{f}(x)=\cos(mx)+b \sum_{k=0}^{m-1}b_k \cos(kx),
\end{equation}
Comparing $\tilde{f}(x)$ with $f_c(x)$ of \eqref{eqpcos} and realizing that $0<b_k<b_{k+1}<1$ we conclude that it is now suffice to prove that $0<b<1$. If we pick an integer $m\ge 11$ in \eqref{bm}, then $0<b<1$. This proved (C).  
\end{proof}
\noindent Figure 6 below showed comparison of $\Phi(t)$(Red) with $\Phi_S(m=11;t)$(Green), and $\Phi_S(m=100;t)$(Blue). 
%
\begin{figure}[H]
\centering
\includegraphics[scale=0.8,keepaspectratio]{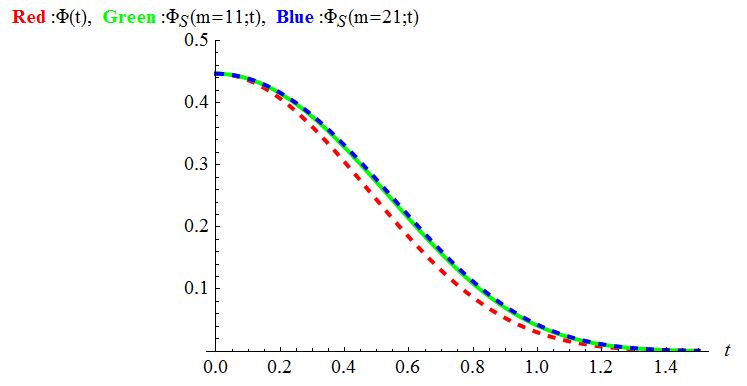}
\caption{Plots of various Phi functions vs. $t$. This includes $\Phi(t)$(Red), $\Phi_S(m=11;t)$ $(b=0.919830)$ (Green), $\Phi_S(m=100;t)$ $(b=0.101181)$ (Blue). }\label{figure2}
\end{figure}
\noindent To quantify the goodness of the approximation, we define and numerically calculate the following relative differences in percentage: 
\begin{equation}
\frac{\int_0^{\infty}|\Phi(t)-\Phi_S(m=11;t)|\mathrm{d}t}{\int_0^{\infty}\Phi(t)\mathrm{d}u}\approx 7.949691 \%
\end{equation}
\begin{equation}
\frac{\int_0^{\infty}|\Phi(t)-\Phi_S(m=100;t)|\mathrm{d}t}{\int_0^{\infty}\Phi(t)\mathrm{d}u}\approx 9.091720 \%
\end{equation}

In Figure 7 below we compare $\Xi_{S}(m=11;z)$(Purple) against $\Xi(z)$(Red). It showed that there existed 29 zeros for both $\Xi_{S}(m=11;z)$ and $\Xi(z)$.
%
\begin{figure}[H]
\centering
\includegraphics[scale=0.8,keepaspectratio]{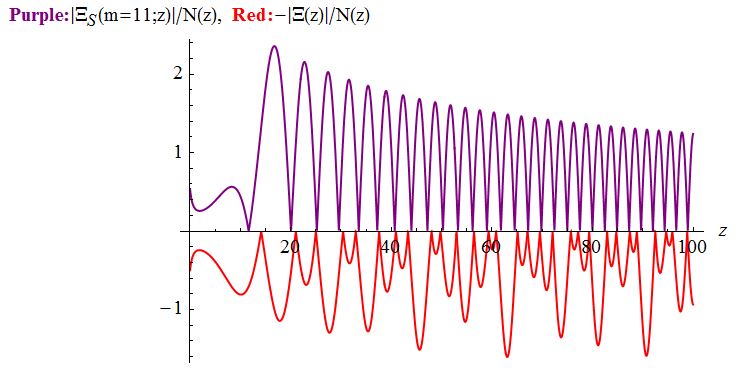}
\caption{Plots of various Xi functions vs. $z$. This includes: $|\Xi_{S}(m=11;z)|/|N(z)|$(Purple); $-|\Xi(z)|/|N(z)|$(Red).}
\end{figure}
In Figure 8 below we compare $\Xi_{S}(m=21;z)$(Blue) against $\Xi(z)$(Red). It showed that there existed 29 zeros for both $\Xi_{S}(m=21;z)$ and $\Xi(z)$.
\begin{figure}[H]
%
\centering
\includegraphics[scale=0.8,keepaspectratio]{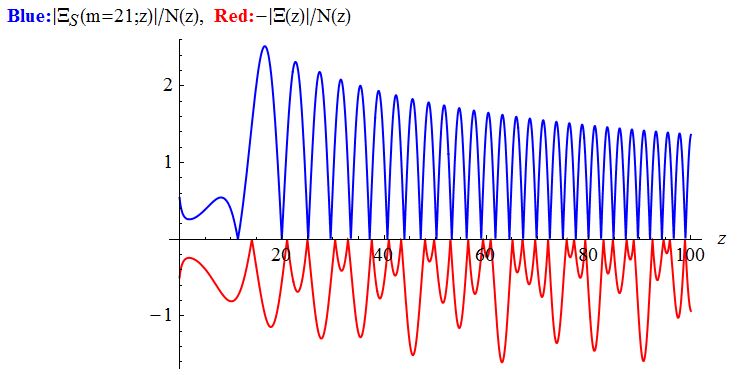}
\caption{Plots of various Xi functions vs. $z$. This includes: $|\Xi_{S}(m=21;z)|/N(z)$(Blue); $-|\Xi(z)|/N(z)$(Red).}
\end{figure}
%
%
%
%
\indent We next approximate $\Phi(t)$ with $\Phi_{S2}(m,a;t)$.
\begin{theorem}
Let
\begin{equation}\label{PhiS2m}
\Phi_{S2}(m,a;t)=4\pi^2 g(m,a;u)\exp\left(-2\pi \cosh t\right),
\end{equation}
\begin{equation}
g(m,a;u)=\cosh\left(9t/4\right)+c \sum_{k=0}^{m-1}c_k(a) \cosh\left(9kt/(4m)\right),
\end{equation}
where $0<c_k(a)=1-a^{m+1}<1,0<a<1,\qquad k=0,1,\dots m-1$.

The Fourier transform of $\Phi_{S2}(m,a,t)$ is:
\begin{equation}\label{XiS2m}
\aligned
\Xi_{S2}(m,a;2z)&=4\pi^2\left(G_{2\pi}(iz+9/4)
+G_{2\pi}(iz-9/4)\right)\\
&+4\pi^2 c \sum_{k=0}^{m-1}c_k(a)(G_{2\pi}(iz+9k/(4m))
+G_{2\pi}(iz-9k/(4m))),
\endaligned
\end{equation}

For a given parameter $0<a<1$, if $\mu,m$ satisfies the equations: 
\begin{equation}\label{mu}
\mu (1-a^\mu)=\beta,
\end{equation}
\begin{equation}\label{mceil}
m \ge \lceil\mu\rceil,
\end{equation}
where the constant $\beta$ is defined in \eqref{bm},
Then 

(A)$\Phi_{S2}(m,a;t)\to \Phi(t),\text{ when}\quad t\to\infty$;

(B)$\Phi_{S2}(m,a;0)= \Phi(0)$; 

(C)the entire function $\Xi_{S2}(m,a;z)$ has only real zeros.
\end{theorem}
\begin{proof}
The proof is similar to that for theorem 1.
Since $0<\frac{k}{m}<1$, $g(m,a;t)\to \cos(9t/4)$ and $\Phi_{S2}(m;t)\to \Phi(t)$ when $t\to\infty$. Thus we proved (A). Setting $\Phi_{S2}(m,a;0)=\Phi(0)$ leads:
\begin{equation}\label{cm}
c =\frac{(e^{2\pi}\Phi(0)-4\pi^2)}{4\pi^2}\frac{(1-a)}{(m(1-a)-a(1-a^m))}=
\frac{\beta}{m-\sum_{k=1}^{m}a^k}>0.
\end{equation}

\noindent If we determined $c$ from \eqref{cm}, then we proved (B).

Becuase of \cref{Kiz},it suffice to prove that $g(m,a;t)$ is a universal factor, or $g(m,a;it)\in \mathcal{LP}$. Defining $x=\frac{9t}{4m}$, then we obtain:
\begin{equation}
g(m,a;it)=\tilde{g}(x)=\cos(mx)+c \sum_{k=0}^{m-1}c_k \cos(kx),
\end{equation}
Comparing $\tilde{g}(x)$ with $f_c(x)$ of \eqref{eqpcos} and realizing that $0<c_k(a)<c_{k+1}(a)<1$ we conclude that it is now suffice to prove that $0<c<1$. 

If we want that $c<1$, then we need to require
\begin{equation}\label{mB2}
m-\sum_{k=1}^{m}a^k>\beta
\end{equation}

\noindent Since $0<a<1$, we have $\sum_{k=1}^{m}a^k>ma^m$. Thus we may require $m$ to satisfy

\begin{equation}\label{mB3}
m(1-a^m) >m-\sum_{k=1}^{m}a^k>\beta
\end{equation}

If we select $m$ that satisfies \eqref{mceil}, then \eqref{mB3} and \eqref{mB2} are satisfied. Thus the parameter $c$ determined by \eqref{cm} satisfies $0<c<1$. Thus we proved (C).    
\end{proof}
\noindent Figure 9 below showed comparison of $\Phi(t)$(Red) with $\Phi_{S2}(m=6,a=1/100;t)$(Green), and $\Phi_{S2}(m=7,a=1/2;t)$(Blue). 
%
\begin{figure}[H]
\centering
\includegraphics[scale=0.8,keepaspectratio]{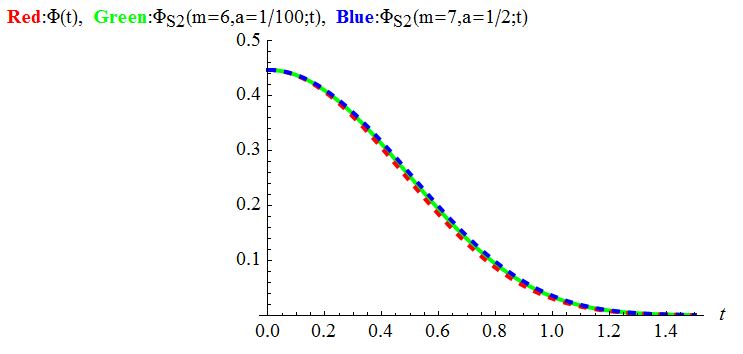}
\caption{Plots of various Phi functions vs. $t$. This includes $\Phi(t)$(Red), $\Phi_S(m=6,a=1/100;t)$ $(b=0.844600)$(Green), $\Phi_S(m=7,a=1/2;t)$ $(b=0.842081)$(Blue). }\label{figure3}
\end{figure}
\noindent The relative differences in percentage: 
\begin{equation}
\frac{\int_0^{\infty}|\Phi(t)-\Phi_{S2}(m=6,a=1/100;t)|\mathrm{d}t}{\int_0^{\infty}\Phi(t)\mathrm{d}u}\approx 2.497861 \%
\end{equation}
\begin{equation}
\frac{\int_0^{\infty}|\Phi(t)-\Phi_{S2}(m=7,a=1/2;t)|\mathrm{d}t}{\int_0^{\infty}\Phi(t)\mathrm{d}t}\approx 3.916332 \%
\end{equation}
In Figure 10 below we compare $\Xi_{S2}(m=6,a=1/100;z)$(Purple) against $\Xi(z)$(Red). It showed that there existed 29 zeros for both $\Xi_{S2}(m=6.a=1/100;z)$ and $\Xi(z)$.
%
\begin{figure}[H]
\centering
\includegraphics[scale=0.8,keepaspectratio]{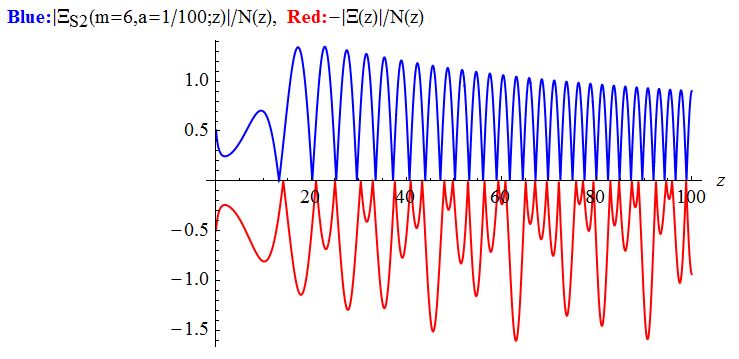}
\caption{Plots of various Xi functions vs. $z$. This includes: $|\Xi_{S2}(m=6,a=1/100;z)|/N(z)$(Purple); $-|\Xi(z)|/N(z)$(Red).}
\end{figure}
In Figure 11 below we compare $\Xi_{S2}(m=7,a=1/2;z)$(Blue) against $\Xi(z)$(Red). It showed that there existed 29 zeros for both $\Xi_{S2}(m=7,a=1/2;z)$ and $\Xi(z)$.
\begin{figure}[H]
%
\centering
\includegraphics[scale=0.8,keepaspectratio]{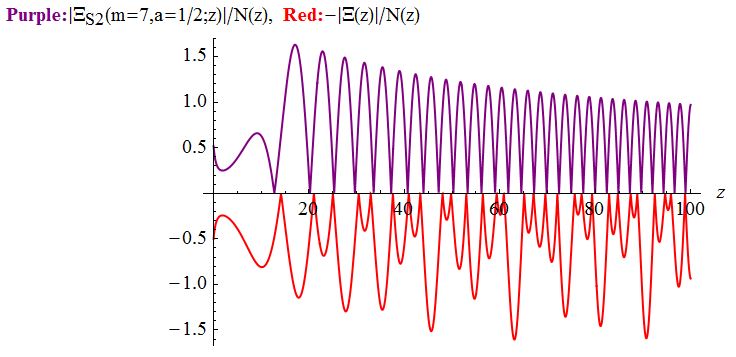}
\caption{Plots of various Xi functions vs. $z$. This includes: $|\Xi_{S2}(m=7,a=1/2;z)|/N(z)$(Blue); $-|\Xi(z)|/N(z)$(Red).}
\end{figure}
%
%
%
%
%
\begin{lemma}~\cite{11}
Let $k\in\N$, then 
\begin{equation}
\sinh^{2k}(x)=\sum_{j=0}^k a_{k,j} \cosh(2jx)
\end{equation}
\noindent where $a_{k,0} = (-1)^k 2^{-2k}\binom {2k}{k}, \quad a_{k,k} = 2^{1-2k}$,
$a_{k,j} = (-1)^{k-j} 2^{1-2k}\binom{2k}{k-j},\quad 0<j<k$.
\end{lemma}

\begin{lemma}
Let $m\in\N$, then  
\begin{equation}
\aligned
h_{m}(t)&=4^m\cosh(t/4)(\sinh^2(t/m)+1-a^2)^{m}\\
&=\sum_{j=0}^{m}b_{m,j}(\cosh(2jt/m+t/4)+\cosh(2jt/m-t/4)).
\endaligned
\end{equation}
\noindent where
\begin{equation}
b_{m,j}=\frac{1}{2}4^m\sum_{k=j}^m\binom{m}{k}(1-a^2)^{m-k}a_{k,j} 
\end{equation}
\end{lemma}
\begin{proof}

\begin{equation}
\aligned
h_{m}(t)
&=4^m\cosh(t/4)(\cosh^2(t/m)-a^2)^{m}\\
&=4^m\cosh(t/4)(\sinh^2(t/m)+1-a^2)^{m}\\
&=4^m\cosh(t/4)\sum_{k=0}^{m}\binom{m}{k}(1-a^2)^{m-k}\sinh^{2k}(t/m)\\
&=4^m\sum_{k=0}^{m}\sum_{j=0}^k\binom{m}{k}(1-a^2)^{m-k} a_{k,j} \cosh(t/4)\cosh(2jt/m)\\
&=\frac{1}{2}4^m\sum_{k=0}^{m}\sum_{j=0}^k\binom{m}{k}(1-a^2)^{m-k}a_{k,j}\\ &\qquad\qquad\qquad\qquad\times\left(\cosh(2jt/m+t/4)+\cosh(2ju/m-t/4)\right)\\
&=\frac{1}{2}4^m\sum_{j=0}^{m}\sum_{k=j}^m\binom{m}{k}(1-a^2)^{m-k}a_{k,j}\\ &\qquad\qquad\qquad\qquad\times\left(\cosh(2jt/m+t/4)+\cosh(2jt/m-t/4)\right)\\
&=\sum_{j=0}^{m}b_{m,j}(\cosh(2jt/m+t/4)+\cosh(2jt/m-t/4)).
\endaligned
\end{equation}
\end{proof}

\indent We next approximate $\Phi(t)$ with $\Phi_{S3}(m;t)$.

\begin{theorem}. Let
\begin{equation}\label{PhiS3m}
\Phi_{S3}(m;t)=2\pi^2 h_{m}(t)\exp\left(-2\pi \cosh t\right),
\end{equation}
\begin{equation}
h_{m}(t)=\cosh(t/4)(4\cosh^2(t/m)-4a^2)^{m},
\end{equation}
where $a\in \R,m\in \N$.

The Fourier transform of $\Phi_{S3}(m;t)$ is:
\begin{equation}\label{XiS3m}
\aligned
\Xi_{S3}(m;2z)&=2\pi^2\sum_{k=0}^{m}b_{m,k}(G_{2\pi}(iz+2k/m+1/4)+G_{2\pi}(iz-2k/m-1/4))\\
&+2\pi^2\sum_{k=0}^{m}b_{m,k}(G_{2\pi}(iz+2k/m-1/4)+G_{2\pi}(iz-2k/m+1/4))
\endaligned
\end{equation}

If parameter $a$ is determined by 

\begin{equation}\label{am}
a =\left(1-\frac{1}{4}(1+\beta)^{1/m}\right)^{1/2}
\end{equation}
where the constant $\beta$ is defined in \eqref{bm},
\noindent then 

(A)$\Phi_{S3}(m;t)\to \Phi(t),\quad \text{when }t\to\infty$;

(B)$\Phi_{S3}(m;0)= \Phi(0)$; 

(C)the Fourier transform of  $\Phi_{S3}(m;t)$ has only real zeros.
\end{theorem}
\begin{proof}
Since
\begin{equation}
\cosh(t/4)\left(4\cosh^2(t/m)-4a^2\right)^{m}\to \cosh(9t/4),\quad \text{when }t\to\infty,
\end{equation}
We proved (A).

\indent Setting $\Phi_{S3}(m;0)=\Phi(0)$ leads to \eqref{am}, thus we proved (B).
 
Becuase of \cref{Kiz},it suffice to prove that $h_{m}(t)$ is a universal factor, or $h_{m}(it)\in \mathcal{LP}$. Since:
\begin{equation}
h_{m}(it)=\tilde{h}(t)=\cos(t/4)\left(2\cos(t/m)-2a\right)^{m}\left(2\cos(t/m)+2a\right)^{m},
\end{equation}
and $\cos(t/4)\in \mathcal{LP}$, it is suffice to prove that $0 < a < 1$.

\noindent When $m=2$, $a^2\approx 0.384621<1$; when $m\to\infty$, $a^2\to 1^{-}$. Since
\begin{equation}
\frac{\mathrm{d}a^2}{\mathrm{d}m}=\frac{(1+\beta)^{1/m}\log(1+\beta)}{4m^2}>0,  
\end{equation}
\noindent $a^2(m)$ is monotonically increasing from $ 0.384621$ to $1^{-}$ when $m$ varies in the range $[2,\infty)$.  Therefore $0<a^2<1$ and $0<a<1$. Thus $\tilde{h}(t)\in \mathcal{LP}$. This proved (C).  
\end{proof}
\noindent Because these Phi functions are so close to each other, we can not readily see their differences in a figure like Figures 1,2,3. So in Figure 4 below we showed comparison of differences: $\Phi_{S3}(m=2;t)-\Phi(t)$,$(b=0.620177)$(Green); $\Phi_{S3}(m=3;t)-\Phi(t)$,$(b=0.737722)$(Blue). 
%
%
\begin{figure}[H]
\centering
\includegraphics[scale=0.8,keepaspectratio]{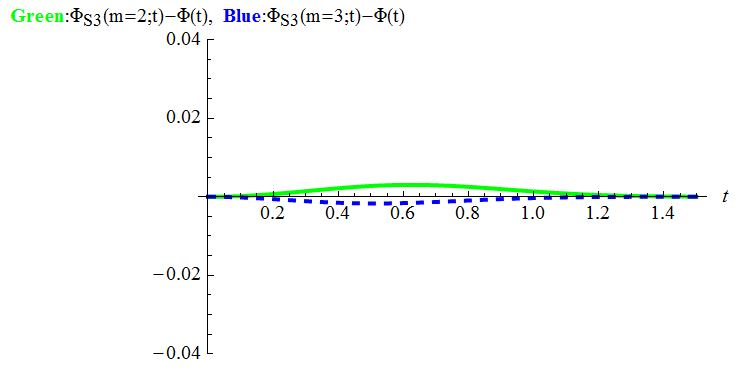}
\caption{Plots of various differences among Phi functions vs. $u$. This includes: $\Phi_{S3}(m=2;t)-\Phi(t)$,$(b=0.620177)$(Green); $\Phi_{S3}(m=3;t)-\Phi(t)$,$(b=0.737722)$(Blue).}\label{figure4}
\end{figure}
\noindent The relative differences in percentage are: 
\begin{equation}
\frac{\int_0^{\infty}|\Phi(t)-\Phi_{S3}(m=2;t)|\mathrm{d}t}{\int_0^{\infty}\Phi(t)\mathrm{d}t}\approx 0.835144 \%
\end{equation}
\begin{equation}
\frac{\int_0^{\infty}|\Phi(t)-\Phi_{S3}(m=3;t)|\mathrm{d}t}{\int_0^{\infty}\Phi(t)\mathrm{d}t}\approx 0.402822 \%
\end{equation}
In Figure 13 below we compare $\Xi_{S3}(m=2;z)$(Blue) against $\Xi(z)$(Red). It showed that there existed 29 zeros for both $\Xi_{S3}(m=2;z)$ and $\Xi(z)$.
\begin{figure}[H]
%
\centering
\includegraphics[scale=0.8,keepaspectratio]{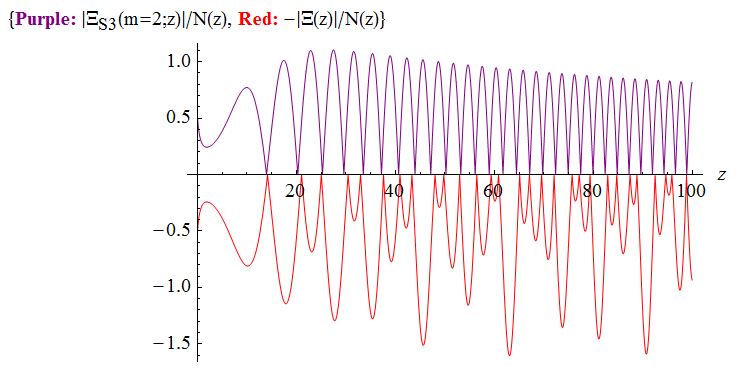}
\caption{Plots of various Xi functions vs. $z$. This includes: $|\Xi_{S3}(m=2;z)|/|N(z)|$(Blue); $-|\Xi(z)|/|N(z)|$(Red).}
\end{figure}

In Figure 14 below we compare $\Xi_{S3}(m=3;z)$(Blue) against $\Xi(z)$(Red). It showed that there existed 29 zeros for both $\Xi_{S3}(m=3;z)$ and $\Xi(z)$.
%
\begin{figure}[H]
\centering
\includegraphics[scale=0.8,keepaspectratio]{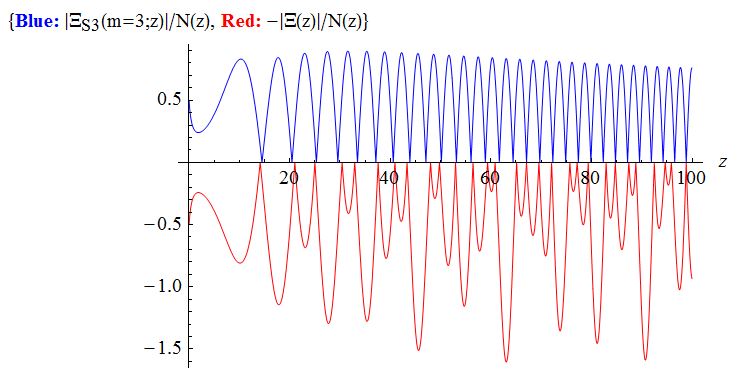}
\caption{Plots of various Xi functions vs. $z$. This includes: $|\Xi_{S3}(m=3;z)|/|N(z)|$(Blue); $-|\Xi(z)|/|N(z)|$(Red).}
\end{figure}

%

%
%

\begin{lemma}
Let $m\in\N$, then  
\begin{equation}
\aligned
j_{m}(t)&=\cosh(t/4)(4\sinh^2(t/(2m))+4a)^{m}(4\sinh^2(t/(2m))+4b)^{m}\\
&=\sum_{j=0}^{m}\sum_{l=0}^{m}d_{m,j,l}p_{m,j,l}(t)
\endaligned
\end{equation}
\noindent where
\begin{equation}
\aligned
d_{m,j,l}&=4^{2m-1}\sum_{k=j}^m\sum_{n=l}^m\binom{m}{k}\binom{m}{n}a^{m-k}b^{m-n} a_{k,j} a_{n,l}\\
p_{m,j,l}(t)&=4\cosh(t/4)\cosh(jt/m)\cosh(lt/m)\\
&=\cos\left(\frac{t}{4}+(+j+l)\frac{t}{m}\right)
+\cos\left(\frac{t}{4}+(+j-l)\frac{t}{m}\right)\\
&+\cos\left(\frac{t}{4}+(-j+l)\frac{t}{m}\right)
+\cos\left(\frac{t}{4}+(-j-l)\frac{t}{m}\right).
\endaligned
\end{equation}
\end{lemma}
\begin{proof}

\begin{equation}
\aligned
j_{m}(t)&=\cosh(t/4)(4\sinh^2(t/(2m))+4a)^{m}(4\sinh^2(t/(2m))+4b)^{m}\\
&=4^{2m}\cosh(t/4)\sum_{k=0}^{m}\binom{m}{k}a^{m-k}\sinh^{2k}(t/(2m))
\sum_{n=0}^{m}\binom{m}{n}b^{m-n}\sinh^{2l}(t/(2m))\\
&=4^{2m}\sum_{k=0}^{m}\sum_{n=0}^{m}\sum_{j=0}^k\sum_{l=0}^n\binom{m}{k}\binom{m}{n}a^{m-k}b^{m-n} a_{k,j} a_{n,l}\\
&\qquad\qquad\qquad\qquad\qquad\times\left(\cosh(t/4)\cosh(jt/m)\cosh(lt/m)\right)\\
&=:4^{2m-1}\sum_{k=0}^{m}\sum_{n=0}^{m}\sum_{j=0}^k\sum_{l=0}^n\binom{m}{k}\binom{m}{n}a^{m-k}b^{m-n} a_{k,j} a_{n,l}p_{m,j,l}(t)\\
&=4^{2m-1}\sum_{j=0}^{m}\sum_{l=0}^{m}\sum_{k=j}^m\sum_{n=l}^m\binom{m}{k}\binom{m}{n}a^{m-k}b^{m-n} a_{k,j} a_{n,l}p_{m,j,l}(t)\\
&=\sum_{j=0}^{m}\sum_{l=0}^{m}\left(4^{2m-1}\sum_{k=j}^m\sum_{n=l}^m\binom{m}{k}\binom{m}{n}a^{m-k}b^{m-n} a_{k,j} a_{n,l}\right)p_{m,j,l}(t)\\
&=:\sum_{j=0}^{m}\sum_{l=0}^{m}d_{m,j,l}p_{m,j,l}(t)\\
\endaligned
\end{equation}
\end{proof}

\indent We next approximate $\Phi(t)$ with $\Phi_{S4}(m,t)$.

\begin{theorem}
Let
\begin{equation}\label{PhiS4m}
\Phi_{S4}(m,u)=2\pi^2 j_{m}(t)\exp\left(-2\pi \cosh t\right),
\end{equation}
\begin{equation}
j_{m}(t)=\cosh(t/4)(4\sinh^2(t/(2m)+4a)^{m}(4\sinh^2(t/(2m)+4b)^{m},
\end{equation}
where $a<b;m\in \N$.
The Fourier transform of $\Phi_{S4}(m;t)$ is:
\begin{equation}\label{XiS4m}
\aligned
\Xi_{S4}(m;2z)&=2\pi^2\sum_{j=0}^{m}\sum_{l=0}^{m}d_{m,j,l}P(m,j,l;z)\\
P(m,j,l;z)&=(G_{2\pi}(iz+1/4+2(+j+l)/m)+G_{2\pi}(iz-1/4-2(+j+l)/m))\\
&=(G_{2\pi}(iz+1/4+2(+j-l)/m)+G_{2\pi}(iz-1/4-2(+j-l)/m))\\
&=(G_{2\pi}(iz+1/4+2(-j+l)/m)+G_{2\pi}(iz-1/4-2(-j+l)/m))\\
&=(G_{2\pi}(iz+1/4+2(-j-l)/m)+G_{2\pi}(iz-1/4-2(-j-l)/m)).
\endaligned
\end{equation}

If $m=2,3$ and the parameters $-1<a<b<1$ are determined by 

\begin{equation}\label{abm1}
\Phi(0)=\Phi_{S4}(m,0)=\pi^2 e^{-2\pi}2^{3+4m}(ab)^m
\end{equation}
\begin{equation}\label{abm2}
\Phi''(0)=\Phi_{S4}''(m,0)= -\pi^2 e^{-2\pi} m^{-1}2^{1+4m}(ab)^{m-1}(abm(32\pi-1)-8(a+b)) 
\end{equation}

\noindent then 

(A)$\Phi_{S4}(m;t)\to \Phi(t),\text{ when }t\to\infty$;

(B)$\Phi_{S4}(m;0)= \Phi(0)$,\qquad $\Phi_{S4}''(m;0)= \Phi''(0)$; 

(C)the Fourier transform of  $\Phi_{S4}(m;t)$ has only real zeros.
\end{theorem}
\begin{proof}
When $t\to\infty$,
\begin{equation}
\cosh(t/4)\left(4\sinh^2(t/(2m)+4a\right)^{m}\left(4\sinh^2(t/(2m)+4b\right)^{m}\to \cosh(9t/4).
\end{equation}
We proved (A).

\indent Setting $\Phi_{S4}(m;0)=\Phi(0)$ and $\Phi_{S4}''(m;0)=\Phi''(0)$ leads to \eqref{abm1} and \eqref{abm2}, thus we proved (B).
 
Becuase of \cref{Kiz},it suffice to prove that $j_{m}(t)$ is a universal factor, or $j_{m}(it)\in \mathcal{LP}$. Since:
\begin{equation}
j_{m}(it)=\tilde{j}(t)=\cos(t/4)\left(4a-4\sin^2(t/(2m)\right)^{m}\left(4b-4\sin^2(t/(2m)\right)^{m},
\end{equation}
and $\cos(t/4)\in \mathcal{LP}$, it is suffice to prove that $0 < a < b < 1$.
From \eqref{abm1} and \eqref{abm2}, we obtain:
\begin{equation}\label{aplusb}
a+b=m\gamma \delta^{1/m},\qquad ab=(1/16)\delta^{1/m}
\end{equation}
\noindent where
\begin{equation}
\aligned
\gamma=\frac{4\Phi''(0)+\Phi(0)(32\pi-1)}{128\Phi(0)}\approx 0.192369\\
\delta=2^{-3}\pi^{-2}e^{2\pi}\Phi(0)\approx 6.059069
\endaligned
\end{equation}
From \eqref{aplusb}, we can solve for $a,b$ and obtain:
\begin{equation}\label{solnab}
\aligned
a(m)&=\frac{1}{2}m\gamma \delta^{1/m}-
\frac{1}{4}\left(\delta^{1/m}\left(4m^2\gamma^2\delta^{1/m}-1\right)\right)^{1/2},\\
b(m)&=\frac{1}{2}m\gamma \delta^{1/m}+
\frac{1}{4}\left(\delta^{1/m}\left(4m^2\gamma^2\delta^{1/m}-1\right)\right)^{1/2}.\\
\endaligned
\end{equation}
\noindent When $m=2$,we find numerically the solution: $a(2)\approx 0.208233<b(2)\approx0.738810<1$. When $m=3$,we find numerically the solution: $a(3)\approx 0.122579<b(3)\approx0.929527<1$.  Thus $\tilde{j}(t)\in \mathcal{LP}$. This proved (C). 
\end{proof}
\noindent We also find out that $a(1),b(1)\in \C$ and $0<a(m)<1<b(m),m>3$.

In Figure 15 below we showed comparison of differences: $\Phi_{S4}(m=2;t)-\Phi(t)$,$(a=0.208233,b=0.738810)$(Green); $\Phi_{S4}(m=3;t)-\Phi(t)$,$(a=0.122579,b=0.929527)$(Blue).
%
\begin{figure}[H]
\centering
\includegraphics[scale=0.8,keepaspectratio]{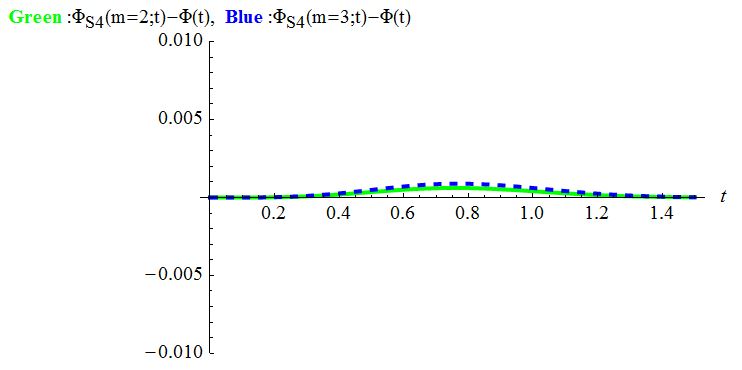}
\caption{Plots of various differences among Phi functions vs. $t$. This includes: $\Phi_{S4}(m=2;t)-\Phi(t)$,$(a=0.208233,b=0.738810)$(Green); $\Phi_{S4}(m=3;t)-\Phi(t)$,$(a=0.122579,b=0.929527)$(Blue).}\label{figure5}
\end{figure}
\noindent The relative differences in percentage are: 
\begin{equation}
\frac{\int_0^{\infty}|\Phi(t)-\Phi_{S4}(m=2;u)|\mathrm{d}t}{\int_0^{\infty}\Phi(t)\mathrm{d}t}\approx 0.835144 \%
\end{equation}
\begin{equation}
\frac{\int_0^{\infty}|\Phi(t)-\Phi_{S4}(m=3;t)|\mathrm{d}t}{\int_0^{\infty}\Phi(t)\mathrm{d}t}\approx 0.402822 \%
\end{equation}
In Figure 16 below we compare $\Xi_{S4}(m=2;z)$(Purple) against $\Xi(z)$(Red). It showed that there existed 29 zeros for both $\Xi_{S4}(m=2;z)$ and $\Xi(z)$.
%
\begin{figure}[H]
\centering
\includegraphics[scale=0.8,keepaspectratio]{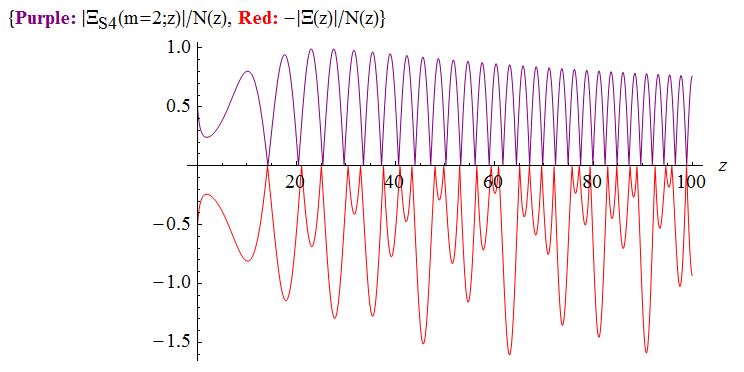}
\caption{Plots of various Xi functions vs. $z$. This includes: $|\Xi_{S4}(m=2;z)|/N(z)$(Purple); $-|\Xi(z)|/N(z)$(Red).}
\end{figure}
In Figure 17 below we compare $\Xi_{S4}(m=3;z)$(Blue) against $\Xi(z)$(Red). It showed that there existed 29 zeros for both $\Xi_{S4}(m=3;z)$ and $\Xi(z)$.
%
\begin{figure}[H]
\centering
\includegraphics[scale=0.8,keepaspectratio]{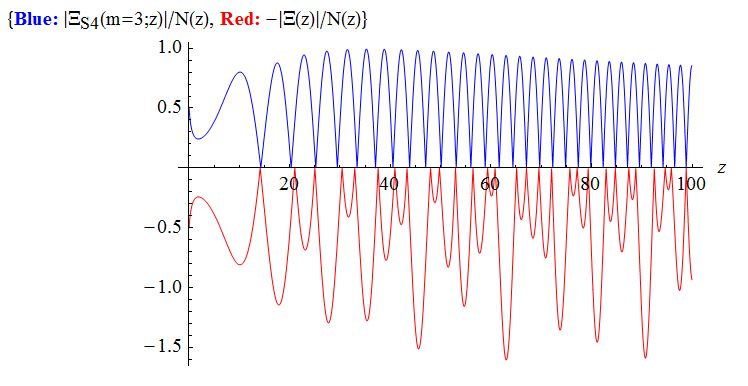}
\caption{Plots of various Xi functions vs. $z$. This includes: $|\Xi_{S4}(m=3;z)|/N(z)$(Blue); $-|\Xi(z)|/N(z)$(Red).}
\end{figure}
%

\section{concluding remarks}

Our criteria for picking kernel $K(t)$ to approximate the kernel $\Phi(t)$ of \eqref{Phi} for Riemann $\Xi(z)$ function are

(i) $K(t)\to \Phi(t),\quad t\to \infty$,

(ii) $K(0)=\Phi(0)$, 

(iii) $\int_0^{\infty}K(t)\cos(z t)\mathrm{d}t$ has only real zeros.

Using these criteria, we obtain several new and improved approximations to kernel $\Phi(t)$ and find out that their Fourier transforms have only real zeros.

Thus this method is quite general and it remains to be seen if one can better approximate the body of the kernel $\Phi(t)$.

\section{acknowledgment}

We appreciate the continuing support and help from Prof. Jie QING (math dept. of University of California at Santa Cruz), Prof. Zixiang ZHOU (math dept. of Fudan University), Prof. Xinyi YUAN, (math dept. of University of California at Berkeley), Dr .Xiaoyi WU.  We greatly appreciate the support and help from Mr. Peter M. Hallum. We appreciate the support and help from the open forums: math.stackexchange.com, tex.stackexchange.com, mathematica.stackexchange.com,  mathoverflow.stackexchange.com.
\pagebreak
%


\begin{thebibliography}{99}
%
\bibitem[1]{1}
E. Bombieri,
\newblock{The Riemann Hypothesis--Official Problem Description.}
\newblock{Clay Mathematics Institute,2001}
%
%
\bibitem[2]{2}
J. Brian Conrey,
\newblock{The Riemann Hypothesis}
\newblock{\it Notices of the American Mathematical Society}
\newblock{341--353 MARCH 2003}
%
\bibitem[3]{3}
D. K. Dimitrov
\newblock{Lee-Yang Measures and Wave Functions}
\newblock{arXiv:1311.0596 [math-ph]}
%
\bibitem[4]{4}
de Bruijn, N. G. 
\newblock{The roots of trigonometric integrals.}, 
\newblock{\it  Duke Math. J} 
\newblock{{\bf 17}, no. 0950 (1950): 197--226}.
%
\bibitem[5]{5}
Dimitar K. Dimitrov and Peter K. Rusev
\newblock{Zeros of Entire Fourier Transforms}
\newblock{ \it East. J. On Approximations}
\newblock{{\bf 17}, no. 1 (2011), 1--108}
%
\bibitem[6]{6}
H. M. Edwards,
\newblock{Riemann?s Zeta Function},  
\newblock{Reprint: Dover Publications},  
\newblock{\it  Academic Press, New York}, 
\newblock{1974}.
%
\bibitem[7]{7}
P. M. Hallum,
\newblock{Zeros of Entire Functions Represented by Fourier Tansforms},
\newblock{Master Thesis,University of Hawai`i at M\={a}noa (2014).} 
%
\bibitem[8]{8}
D. A. Hejhal,
\newblock{On a Result of G. P\'{o}lya concerning the Riemann xi-function}, 
\newblock{\it J. Anal. Math.}, 
\newblock{{bf 55}(1990), 59--95}
%
\bibitem[9]{9}
H. Ki,
\newblock{The zeros of Fourier transforms}
\newblock{\it Fourier series methods in complex analysis}
\newblock{Joensuu Dept. Math. Rep. Ser., \it 10, 113--127}
\newblock{Univ. Joensuu, Joensuu, 2006}
%
\bibitem[10]{10}
Joseph Lehner,
\newblock{Discontinuous Groups and Automorphic Functions},  
\newblock{\it Amer. Math. Soc, Providence, Rhode Island}, 
\newblock{1964}.

%
\bibitem[11]{11}
M. Fisher,
\newblock{answer at mathematics.stackexchange.com},

%
\bibitem[12]{12}
G. Pólya,
\newblock{On the zeros of certain trigonometric integrals}, 
\newblock{\it J. London Math. Soc}, 
\newblock{{bf 1}(1926), 98--99}
\newblock{Reprinted as item [90] in (P1974).}
%
\bibitem[13]{13}
G. Pólya,
\newblock{Bemerkung Über die Darstellung der Riemanschen xi-Funktion},
\newblock{\it  Acta Math.}, 
\newblock{{\bf 48} (1926), 305--317}
\newblock{Reprinted as item (93) in (P1974).}
%
\bibitem[14]{14}
G. Pólya,
\newblock{Uber trigonometrische Integrale mit nur reellen Nullstellen}, 
\newblock{\it J. r. angew. Math.} 
\newblock{{\bf 158} (1927), 6--18}
\newblock{(Reprinted as item (101) in (P1974).)}
%
\bibitem[15]{15}
G. Pólya,
\newblock{Collected Papers},  
\newblock{Volume II},  
\newblock{\it The Clarendon Press, MIT Press,Cambridge, Mass.}, 
\newblock{1974}.
%
\bibitem[16]{16}
G. F. B. Riemann, 
\newblock{Ueber die Anzahl der Primzahlen unter einer gegebenen GrÄosse},
\newblock{\it Monatsber. Konigl. Preuss. Akad. Wiss.}
\newblock{Berlin (1859), 671--680}.
\newblock{English translation: Edwards(E)p. 299?305}
%
\bibitem[17]{17}
E.C. Titchmarsh,
\newblock{The Theory of the Riemann Zeta-Function},  
\newblock{Second edition, Edited and with a preface by D. R. Heath-Brown},  
\newblock{\it The Clarendon Press, Oxford University Press, New York}, 
\newblock{1986}.
%
\bibitem[18]{18}
E. T. Whittaker and G. N. Watson,
\newblock{A Course of Modern Analysis},   
\newblock{\it Cambridge University Press, Cambridge}, 
\newblock{1902}.
%


\end{thebibliography}
\end{document}